\documentclass[11pt,reqno]{amsart}
\usepackage[margin=1in]{geometry}
\usepackage{graphicx,color,amsmath,amssymb,mathrsfs}
\usepackage{amsfonts,amsthm,epsfig,epstopdf, esint,soul}
\usepackage[mathscr]{euscript}
\usepackage{amsthm}

\author{Robin Neumayer}
\newtheorem{theorem}{Theorem}[section]
\newtheorem{lemma}[theorem]{Lemma}
\newtheorem{proposition}[theorem]{Proposition}

\theoremstyle{remark}
\newtheorem{remark}[theorem]{Remark}
\theoremstyle{definition}

\usepackage[utf8]{inputenc}
\usepackage[english]{babel}

\usepackage{xcolor}
\usepackage[colorlinks = true,
            linkcolor = darkgray,
            urlcolor  = darkgray,
            citecolor = darkgray,
            anchorcolor = blue]{hyperref}

\numberwithin{equation}{section}

\newcommand{\F}{{\mathcal{F}}}

\newcommand{\R}{{\mathbb{R}}}
\newcommand{\e}{{\epsilon}}

\newcommand{\Hi}{{\mathcal{H}}}

\newcommand{\Om}{\Omega}
\newcommand{\om}{\omega}
\newcommand{\na}{{\nabla}}
\newcommand{\ETA}{{{\lambda}}}
\newcommand{\W}{{\mathcal{W}}}
\newcommand{\Ric}{{\rm{Ric}}}
\newcommand{\Rm}{{\rm{Rm}}}
\newcommand{\vol}{{\rm{vol}}}

\title[Scalar lower bounds and volume upper bounds]{Epsilon regularity under  scalar curvature and entropy lower bounds and volume upper bounds}
\begin{document}
\maketitle
\begin{abstract}
	Examples show that Riemannian manifolds 
with almost-Euclidean lower bounds on scalar curvature and Perelman entropy need not be close to Euclidean space in any metric space sense. Here we show that if one additionally assumes an almost-Euclidean upper bound on volumes of geodesic balls, then unit balls in such a space are Gromov-Hausdorff close, and in fact bi-H\"{o}lder and bi-$W^{1,p}$ homeomorphic, to Euclidean balls. We  prove a compactness and limit space structure theorem under the same assumptions.
\end{abstract}
\section{Introduction}
The scalar curvature of a Riemannian manifold $(M^n, g)$ governs volumes of geodesic balls of asymptotically small radii. More specifically,  for $x \in M$ and $r>0$, 
\begin{equation}
	\label{eqn: expansion}
\vol_g(B_g(x,r)) = \omega_n r^n \left\{ 1 - \frac{R_g(x)}{6(n+2)} r^2  + O(r^4) \right\}.
\end{equation}
Here $B_g(x,r)$ is a geodesic ball, $\omega_n r^n$ is the volume of a Euclidean ball of radius $r$, and $R_g(x)$ is the scalar curvature of $(M,g)$ at $x.$ For a Riemannian manifold with a lower bound $R_g \geq -\delta$ on the scalar curvature,  this Taylor expansion gives an upper bound \begin{equation}\label{eqn: volume upper bound}
\vol_g(B_g(x,r)) \leq (1+\delta)\omega_n r^n
\end{equation}
for the volumes of balls of sufficiently small radii.  However, the error term $O(r^4)$ in \eqref{eqn: expansion} and hence the threshold of ``sufficiently small'' in \eqref{eqn: volume upper bound} depends not only on the scalar curvature, but on the full curvature tensor of the metric $g$. This means that in practice, \eqref{eqn: expansion} and \eqref{eqn: volume upper bound} have limited utility in the study of spaces with lower bounds on scalar curvature.

On the other hand, lower bounds on the Ricci curvature lead to volumes control for geodesic balls of all radii. In particular, for small $\delta>0$,  a lower bound $\Ric_g \geq -c_n\delta$ on the Ricci curvature  implies that  the upper bound \eqref{eqn: volume upper bound} holds for all scales $r \in (0,1]$ thanks to the Bishop-Gromov inequality. 
Furthermore, if a unit ball in such a space has an almost-Euclidean   ``noncollapsing''  bound  $\vol_g(B_g(x,1)) \geq (1-\delta)\omega_n$, then $B_g(x,r)$ has almost-Euclidean volume for every scale $r \in (0,1]$, again by Bishop-Gromov.  Cheeger and Colding \cite{Colding1,CC1} proved the following epsilon regularity theorem in this context, which serves as a starting point for the structure and regularity theory for Riemannian manifolds with Ricci curvature lower bounds and their limit spaces \cite{CC1, CC2, CC3, CCannals, CN, CNJ}.
\begin{theorem}[Cheeger-Colding]\label{thm: CC}
	Fix $n\geq 2$ and $\e>0$. There exists $\delta=\delta(n,\e)>0$ such that if $(M,g)$ is a Riemannian manifold with $\Ric_g \geq -\delta$ and $\vol_g(B_g(x,1))\geq (1-\delta) \om_n$ for a given $x \in M$, then $d_{GH}(B_{g}(x,1) , B_{g_{euc}}(0^n, 1) )<\e$.
\end{theorem}
Here $d_{GH}$ is the Gromov-Hausdorff distance; see for instance \cite[Chapter 10]{petersenBook}.

In \cite{LNN1}, Lee, Naber and the author sought to formulate and prove an analogous epsilon regularity theorem with lower bounds on scalar curvature in place of Ricci curvature. 
It turns out that the Gromov-Hausdorff control of Theorem~\ref{thm: CC} is false in that context, even when one assumes an almost Euclidean bound $\nu(g, 2)\geq -\delta$ on the Perleman entropy, which is a stronger type of noncollapsing condition. (We define the Perelman entropy precisely and discuss its relation to the volume noncollapsing condition in section~\ref{sec: prelim}.) We construct examples in \cite{LNN1} (see also \cite{LNNSurvey} and \cite{LeeTopping}) showing that Riemannian manifolds with almost-Euclidean lower bounds on scalar curvature and Perelman entropy can have metric space structures that are far from Euclidean.
\begin{theorem}[Lee-Naber-N.]\label{thm: example}
	Fix $n \geq 4$. There exists $C=C(n)$ such that for every $\delta>0$ there exists a complete Riemannian manifold with bounded curvature such that $R_g \geq -\delta$ and $\nu(g, 2)\geq -\delta$, but $d_{GH}(B_g(x,1), B_{euc}(0^n, 1)) \geq C$.
\end{theorem}
 The simplest example of  Theorem~\ref{thm: example} comes from  a family of metrics $g_i$ on $\R^n$ that are Euclidean away from a line $\{ x_1=\dots=x_{n=1} =0\}$ and become increasingly degenerate along this line; in the pointed Gromov-Hausdorff limit the entire line collapses to a point.
 In a related example, we paste copies of this family onto a flat torus to construct a sequence of metrics $g_i$ on the torus whose volumes converge to the volume $v$ of the initial flat torus but that converge in the Gromov-Hausdorff topology to a single point.
 
  In both of these examples, the ``sufficiently small'' scale up to which \eqref{eqn: volume upper bound} holds degenerates to zero along the sequence. For the first, the volumes $\vol_{g_i}(B_{g_i}(0^n, r))$ tend to infinity as $i \to \infty$ for any fixed $r$; in the global Euclidean chart in which the $g_i$ are defined, $g_i$-geodesic balls look increasingly elongated along the central fiber.  In the second, the volumes of $g_i$-balls of any radius converge to $v$ as $i \to \infty.$  
  
In this note we show that if the upper bound \eqref{eqn: volume upper bound} holds up to a definite scale, then distance functions cannot degenerate. With this additional assumption, we recover  Gromov-Hausdorff a priori regularity analogous to Theorem~\ref{thm: CC}, and in fact bi-H\"{o}lder and bi-$W^{1,p}$ estimates,  as well as measured Gromov-Hausdorff compactness and several structural properties of limit spaces. We begin with the epsilon regularity theorem:

\begin{theorem}\label{thm: GH}
Fix $n \geqslant 2$ and $\e>0$, $\alpha \in (0,1)$ and $p\geq 1$. There exists $\delta=\delta(n, \e, \alpha, p)>0$ such that the following holds. 
 Let (M,g) be a complete Riemannian $n$-manifold with bounded curvature such that
\begin{align}
\label{e: scalar and entropy lower bound}
R_g \geq-\delta ,& \qquad \nu(g,2) \geq-\delta 
\end{align} 
and 
\begin{align}\label{en: vol upper bound theorem statement}
\vol_g(B_g(x, r)) \leq\left(1+\delta\right) \omega_n r^n& \qquad \forall r \in(0,2] .	
\end{align}
	 Then for any $x \in M$, 
	 there is a smooth diffeomorphism 
	 $\psi: B_g(x,1)\to U \subset \R^n$ with $\psi(x)=0$ and $B_{g_{euc}}(0,1-\e)\subset U\subset B_{g_{euc}}(0,1+\e)$ satisfying the bi-H\"{o}lder estimates 
	 	 \begin{align}\label{eqn: biholder ests th1}
	 	(1-\e)d_g(y,z)^{1/\alpha}\leq |\psi(y)-\psi(z)| \leq (1+\e)d_g(y,z)^\alpha
	 \end{align}
for all $y,z\in B_g(x,1)$ 	and the $W^{1,p}$ estimates
	 \begin{align}\label{eqn: w1p 1 intro}
  \fint_{U} \left|(\psi^{-1})^*g -g_{euc} \right|^p \, dy \leq \e, \qquad
 \fint_{B_g(x,1)} \left|\psi^* g_{euc} - g\right|^p \, d\vol_g \leq \e
	\end{align}
\end{theorem}

The assumption \eqref{e: scalar and entropy lower bound} implies that  ${ \vol_g(B_g(x,r))}\geq (1 - \e)\omega_n r^n $ for all $x \in M$ and $r \in (0,1]$, i.e. balls are non-collapsing up to scale one with an almost-Euclidean volume ratio; see Lemma~\ref{lemma: volume noncollapsing}. In particular, the assumptions \eqref{e: scalar and entropy lower bound} and \eqref{en: vol upper bound theorem statement} together imply that the volume measure is doubling up to scale one. In \eqref{eqn: w1p 1 intro}, $|\cdot|$ denotes the tensor norm with respect to $g$ and $g_{euc}$ respectively. The estimate \eqref{eqn: w1p 1 intro} is equivalent to $\fint_{B_g(x,1)} |d \psi|^p\,d\vol_g \leq 1+\e$ and $\fint_{U}  |d\psi^{-1}|^p\,dy \leq 1+\e,$ i.e. they are $W^{1,p}$ estimates for the map $\psi$ and its inverse. 
Theorem~\ref{thm: GH} in particular implies that for any $x \in M$, we have
$d_{G H}\left(B_g(x, 1), B_{g_{euc}}(0^n, 1)\right) \leq \e\,.$

Theorem~\ref{thm: GH} is closely related to the work of Bing Wang in \cite{BingWangB}, and in fact the first conclusion \eqref{eqn: biholder ests th1} can be deduced from results there. In \cite{BingWangB} and the companion paper \cite{BingWangA}, Wang introduces and proves localized versions of many of Perelman's fundamental Ricci flow  concepts, included entropy functionals, no-local collapsing theorems, and pseudo-locality theorems. In \cite[Section 5]{BingWangB}, he points out the key ideas, which are already present in \cite{TianWang}, that ``when the volume element is decreasing and the distance is expanding, we shall have a rough distance distortion estimate along the [Ricci] flow, if the initial volume ratio has an upper bound,'' and that such an estimate can be refined using an integral estimate for the scalar curvature. The conclusion \eqref{eqn: biholder ests th1} in Theorem~\ref{thm: GH} can be shown by combining Proposition 5.3 and Theorem 5.4 of Wang's \cite{BingWangB} with Lemma~\ref{lemma: initial RF} below. Instead we give a slightly different (though fundamentally similar) proof, which is self contained apart from a few facts pulled from \cite{LNN1}. The conclusion~\eqref{eqn: w1p 1 intro} does not follow from \cite{BingWangB} and rests upon a decomposition theorem proven by Lee, Naber and the author in \cite{LNN1}. The overall proof scheme of Theorem~\ref{thm: GH} is similar to the proof of the main epsilon regularity theorem in  \cite[Theorem 1.1]{LNN1}.

Next, we have a  compactness result and structure theorem for limit spaces under the same assumptions of Theorems~\ref{thm: GH}.
\begin{theorem}\label{thm: limit structure theorem}
	Fix $n\geq 2$, $\e> 0$, $C_0>0$, $r_0>0$, and $\tau_0>0$.
There exists $\delta=\delta(n, \e )>0$ such that the following holds. Let 
$\{ (M_i, g_i,x_i)\}$
	be a sequence of complete pointed Riemannian manifolds with bounded curvature satisfying 
	\begin{align}\label{eqn: hp limit}
		R_{g_i}\geq -C_0 ,\qquad		\nu(g_i, \tau_0) \geq -\delta \,, \qquad \text{and} \qquad \vol_{g_i}(B_{g_i}(x, r) ) \leq (1+\delta) \om_n r^n
	\end{align}
	for all $x \in M_i$ and $r \in (0,r_0]$.
	Then, up to a subsequence, $(M_i, g_i,x_i, d\vol_{g_i})$ converges in
the pointed measured Gromov-Hausdorff topology to a pointed metric measure space $(X, d, x_\infty, m)$ satisfying the following properties:
\begin{enumerate}
	\item $X$ is an $n$-dimensional manifold.
	\item For each $x\in X$ and $r\in(0,\rho_0)$, where $\rho_0^2 := \min\{\delta/C_0,\tau_0/2, r_0^2/4\}$ we have
	\begin{equation}\label{eqn: limit space volumes of balls}
		(1-\e)\om_nr^n\leq m(B_d(y,r))\leq (1+\delta)\om_n r^n.
	\end{equation}
	\item\label{item4} The measure $m$ and the $n$-dimensional Hausdorff measure $\mathcal{H}^n$ are mutually absolutely continuous, with $m =f\mathcal{H}^n$ for a function $f$ satisfying $(1-\e) \leq f \leq (1+\delta)$.
	\item $X$ is $\Hi^n$-rectifiable and $m$-rectifiable. 
\end{enumerate}
\end{theorem}
Recall that a metric space $X$ is $m$-rectifiable if there is a countable collection of $m$-measurable subsets $\mathcal{F}^k\subset X$ and bi-Lipschitz maps $\psi_k:\mathcal{F}^k \to U_k\subset\R^n$ such that $m (X \setminus \cup_k \mathcal{F}^k)  =0.$

Let us make some comments regarding the theorems above.
\begin{enumerate}
\item 
In view of  recent work of Liang Cheng in \cite{LiangCheng}, as well as Wang's work in \cite{BingWangA, BingWangB}, it is likely possible to establish 
 local versions of Theorems~\ref{thm: GH} and \ref{thm: limit structure theorem}. We will pursue this in upcoming work.

\item In Theorem~\ref{thm: limit structure theorem}, one is led to wonder whether the measure $m$ is equal to the $n$-dimensional Hausdorff measure $\mathcal{H}^n$. We are not sure. If one replaces the volume growth assumption in \eqref{eqn: hp limit} with the stronger assumption $\vol_{g_i}(B_{g_i}(x, r)) \leq (1+ \sigma(r))\omega_n r^n$ for all $x \in M_i$ and $r \in (0, r_0]$ for some modulus of continuity $\sigma$ that is uniform in $i$, then indeed one can show that $m = \mathcal{H}^{n}$ using a refinement of the decomposition theorem in Lemma~\ref{lemma: decomposition}. 

\item In \cite[Theorem 1.15]{LNN1}, we prove a different type of compactness and limit structure theorem for sequences satisfying only the first two conditions of \eqref{eqn: hp limit} in Theorem~\ref{thm: limit structure theorem}.  In view of the examples discussed above, the objects that arise as limit spaces are not metric spaces. Instead, we show in \cite{LNN1} for any $p>n$, one can choose $\delta=\delta(n,p)$ small enough so that such sequences converge in the pointed $d_p$ sense to a so-called pointed rectifiable Riemannian space 
 that has certain nice structural properties. Without getting too far into details, we mention  that
 it would be interesting to systematically investigate the relationship and compatibility between the limits of Theorem~\ref{thm: limit structure theorem} and \cite[Theorem 1.15]{LNN1}.

\item In \cite[Theorem 1.11]{LNN1}, we prove that under only assumption \eqref{e: scalar and entropy lower bound}, the metric enjoys $W^{1,p}$ estimates of a similar form to \eqref{eqn: w1p 1 intro} (with the important distinction that in \cite{LNN1} the domain of $\psi$ and of the second integral in \eqref{eqn: w1p 1 intro} is not $B_g(x,1)$ but rather a ``$d_p$ ball''). Theorem~\ref{thm: GH} shows  that with the additional assumption \eqref{en: vol upper bound theorem statement} of control from above on the volumes of balls, these estimates (which give $L^p$ estimates for the metric coefficients) can be upgraded to control on the distance functions. Although their results are not directly applicable here, we point out that similar themes are present in \cite{AllenSormani, BrianAllen}, where various conditions are given allowing one can upgrade from $L^p$ control on metric coefficients to control of the distance functions, and in \cite{APS} where upper volume control combined with one-sided control of distance functions is shown to imply volume preserving intrinsic flat convergence.

\end{enumerate}

\noindent{\it Acknowledgements:} The author is partially supported by NSF Grant DMS-2155054 and the Gregg Zeitlin Early Career Professorship at CMU, and is grateful to Aaron Naber and Man-Chun Lee for many enlightening discussions over the years, and to the latter for pointing out the reference \cite{LiangCheng}. Part of this work was carried out while the author was visiting the Fields Institute.
\section{Preliminaries}\label{sec: prelim}
 The Perelman entropy was introduced by Perelman in \cite{perelman2002entropy} as a monotone quantity for the Ricci flow, and is defined in the following way. For a function $f \in C^\infty(M)$ and real number $\tau>0$, Perelman's $\W$ -functional is defined by
\begin{equation}\label{eqn: W functional}
\W(g,f,\tau) = \frac{1}{(4\pi \tau)^{n/2}}\int_M \left\{ \tau(|\na f|^2 + R_g) +f-n\right\} {e^{-f}} \, d \vol_g  \,.
\end{equation}
The Perelman entropy $\mu(g,\tau)$ is 
\begin{equation}\label{eqn: perelman entropy def}
	\mu(g,\tau) = \inf \left\{ \W(g,f,\tau) : \frac{1}{(4\pi \tau)^{n/2}}\int_M {e^{-f}} \, d\vol_g= 1\right\}\,.
\end{equation}
This quantity can be viewed as the optimal constant in a log-Sobolev inequality on $(M,g)$ at scale $\tau^{1/2}$; on Euclidean space, Gross's log-Sobolev inequality \cite{Gross} is equivalent to the fact that $\mu(g_{euc}, \tau)=0$ for every $\tau >0$.
An important feature of the  Perelman entropy is that it is extremized by Euclidean space: for any complete Riemannian manifold $(M,g)$ with bounded curvature, $\mu(g,\tau) \leq 0$ for every $\tau>0$, and equality holds for some $\tau>0$ if and only if $(M,g)$ is isometric to Euclidean space.   We will assume that, for a given $\delta>0$, the Perelman entropy is $\delta$-close to that of Euclidean space for at all scales below two. This is compactly expressed as $\nu(g,2)\geq -\delta$, where Perelman's $\nu$-functional is defined by
\begin{equation*}
	\nu(g,\tau) = \inf\{ \mu(g,\hat{\tau}): \hat{\tau}\in (0, \tau)\}.
\end{equation*}

Recall that $(M, g(t))_{t \in (0,T)}$ is a solution to the Ricci flow if 
$\partial_t g_t = -2\Ric_{g(t)}.$ 
The scalar curvature evolves along the Ricci flow as $(\partial_t -\Delta_{g(t)})R_{g(t)} = 2|\Ric_{g(t)}|^2$, and in particular as a supersolution to the heat equation, its lower bounds are preserved along the flow. This feature, together with its regularizing properties, make the Ricci flow
useful tool for studying spaces with lower bounds on scalar curvature; see \cite{LNN1, Bamler1, PBG1, PBG2} for instance.  The volume form evolves via $\partial_t d\vol_{g(t)} = -R_{g(t)} d\vol_{g(t)}$ along the Ricci flow. So, if an initial metric $g(0)$ satisfies $R_{g(0)}\geq -\delta$, then the same lower bound persists and volumes do not expand too much along the flow: $
 	d\vol_{g(t)}\leq \exp\{\delta(t-s)\}d\vol_{g(s)}$ for all $s \leq t.$ So, provided $\delta\leq 1/2,$  a Taylor expansion shows that for all $ 0\leq s \leq t\leq \min\{1,T\}$, 
 	\begin{equation}
 	\label{eqn: volume forms}
 	d\vol_{g(t)}\leq \left\{1+2\delta(t-s)\right\}d\vol_{g(s)}. 
 \end{equation} 
The Perelman entropy interacts naturally Ricci flow as well, by design. The following lemma highlights some of the basic Ricci flow facts that will be useful in our setting.

\begin{lemma}\label{lemma: initial RF}
Fix $n \geq 2$,  $\lambda>0$ and $\e>0$. There exists $\delta=\delta(n, \lambda,\e)$ such that if $(M, g)$ is a complete Riemannian manifold with bounded curvature with $\nu(g, 2) \geq-\delta$, then the smooth Ricci flow $(M, g(t))$ with $g(0)=g$ exists for $t \in (0,1]$ and has the scale-invariant curvature estimates
\begin{equation}
	\label{eqn: curvature estimates}
|\Rm_{g(t)}| \leqslant \frac{\lambda}{t} \qquad \text{for all }t \in(0,1] .
\end{equation}
Moreover, for
any $x\in M$ and $t\in(0,1]$, we may find a
	diffeomorphism $\psi:B_{g(t)}(x, 16 t^{1/2})\to \Omega\subset \R^n$ such that $\psi(0)=x$ and
\begin{equation}
		(1-\e) g_{euc}  \leq   \psi^*g(t)\leq (1+\e)g_{euc} 
\end{equation}
for all $y\in \Omega$. 
	In particular, 	for any $r\in (0, 16t^{1/2})$ we have	
	\begin{equation}\label{eqn: under reg scale euclidean volumes}
		(1-\e)\om_n r^n \leq \vol_{g(t)}(B_{g(t)}(x,r))\leq(1+\e)\om_n r^n
	\end{equation}
\end{lemma} 
\begin{proof} The uniform existence time and scale invariant curvature estimates \eqref{eqn: curvature estimates} were proven in \cite[Theorem 3.2]{LNN1}, and the proof there is essentially from Hein and Naber \cite{HeinNaber}. The proof is a contradiction argument using  Shi's derivative estimates \cite{ShiExistence} and the rigidity of Euclidean space as a maximizer of the entropy.

Next, Perelman's no local collapsing theorem implies that balls are noncollapsing below scale $t^{1/2}$ in the sense that  $\vol_{g(t)}(B_{g(t)}(x,r)) \geq \kappa r^n$ for all $x \in M$ and $r \leq t^{1/2} \leq1$ for a uniform constant $\kappa= \kappa(n, \lambda, \delta)$; see \cite{perelman2002entropy}. 
Together with the curvature estimates \eqref{eqn: curvature estimates}, this implies that the injectivity radius of $(M,g(t))$ is bounded below by $i_0 t^{1/2}$ for a number $i_0 = i_0(n, \delta, \lambda)>0$; see \cite[Chapter 10, Lemma 53]{petersenBook}.
 A further contradiction argument using the rigidity of Euclidean space for the entropy as in the proof of \cite[Theorem 3.2]{LNN1} shows that, that up to decreasing $\delta$, we may assume that $i_0 \geq 16 $. In turn, this implies the second part of the lemma, say in normal coordinate charts. Though we won't directly need it here, it is worth noting that Shi's derivative estimates actually tell us that $B_{g(t)}(x, 16t^{1/2})$ is smoothly close to a Euclidean ball. 
\end{proof}
In view of Lemma~\ref{lemma: initial RF},  under the hypotheses of Theorem~\ref{thm: GH} we may assume that the Ricci flow $(M,g(t))$ with $g(0)=g$ exists for all $t \in (0,1]$ and $B_{g(1)}(x, 16)$ is smoothly close to a Euclidean ball for any $x\in M$.
 To prove Theorem~\ref{thm: GH}, then, it suffices prove bi-H\"{o}lder and bi-$W^{1,p}$ estimates for the identity map between $t=0$ and $t=1$. In particular we will need to compare the distance functions along the flow. One direction of this comparison comes essentially for free by  combining the scale-invariant curvature estimates \eqref{eqn: curvature estimates} with 
 the following  one-sided distance distortion estimate due to Hamilton. We state it only in the form needed here; see \cite[Theorem 17.1]{HamiltonRFsurvey} or \cite[Lemma 8.33]{CLNbook}.

\begin{lemma}[Hamilton]\label{prop: half distortion}
Let $(M,g(t))_{t \in (0,1]}$ be a smooth Ricci flow satisfying \eqref{eqn: curvature estimates}. 
  Then for any $x,y\in M$, and $0\leq s \leq t\leq1$ we have 
	 \begin{equation}\label{eqn: one sided bound}
	 d_{g(t)}(x,y) \geq d_{g(s)}(x,y) -8\sqrt{(n-1)\ETA t}.
	 \end{equation}
\end{lemma} 
Let us observe that \eqref{eqn: one sided bound} implies that for any $\rho>0$ and $0\leq s\leq t \leq 1,$  setting $C^2 =64(n-1)$, we have
\begin{equation}\label{eqn: containment of balls}
	B_{g(t)}\left(x, \rho\right) \subseteq B_{g(s)}\big(x, \rho+ \sqrt{C\ETA t}\big).
\end{equation}
By choosing $s=0$ and $\ETA$ (hence $\delta$) sufficiently small depending on $\e$, Lemma~\ref{prop: half distortion} implies 
\begin{equation}
	\label{eqn: one containmentA}
B_{g(t)}(x, \rho) \subset B_{g(0)}(x , (1+\e) \rho)
\end{equation}
for any $\rho \geq t^{1/2}/100.$

Lemma~\ref{prop: half distortion} asserts that distances cannot decrease too much along the flow when \eqref{eqn: curvature estimates} is in force. On the other hand, the evolution of the volume form shows that volumes cannot increase too much along the flow under the lower scalar bound. The interplay between these two estimates implies that under a scalar curvature lower bound, the entropy noncollapsing condition $\nu(g,2)\geq -\delta$ implies the volume noncollapsing condition \eqref{eqn: volume lower bound static} at every point.

\begin{lemma}[Entropy noncollapsing implies volume noncollapsing] 
\label{lemma: volume noncollapsing}Fix $n\geq 2$ and $\e>0$. There exists $\delta =\delta(n,\e)>0$ such that the following holds. 
	Suppose that $(M,g)$ is a complete Riemannian manifold with bounded curvature such that $R_g \geq -\delta$ and $\nu(g, 2)\geq -\delta$. 
	Then for every $r \in (0,1]$,
	\begin{equation}
		\label{eqn: volume lower bound static}
	\vol_{g}(B_{g}(x,r) ) \geq (1-\e) \omega_n r^n\,.
	\end{equation}
More generally, if $(M,g(t))_{t \in(0,1]}$ is the Ricci flow with $g(0)=g$, whose existence is guaranteed by Lemma~\ref{lemma: initial RF}, then for every $t \in [0,1]$ and $r \in (0,1]$,
	\begin{equation}
		\label{eqn: volume lower bound}
	\vol_{g(t)}(B_{g(t)}(x,r) ) \geq (1-\e) \omega_n r^n\,.
	\end{equation}

\end{lemma}
\begin{proof}Fix $t \in [0,1]$ and $r \in (0,1]$.
If $r\leq t^{1/2}$, then \eqref{eqn: volume lower bound} is a  consequence of  \eqref{eqn: under reg scale euclidean volumes}.
Next consider the case when $r \in (t^{1/2}, 1]$. 
We apply \eqref{eqn: containment of balls}, taking $s=t$ and $t=r^2,$ and $\rho =r-\sqrt{C\ETA t}$ where $C$ is the constant in \eqref{eqn: containment of balls}. Provided $\ETA$, and hence $\delta$,  is sufficiently small,  we find that 
$	B_{g(r^2)}(x, (1-\e)r) \subseteq B_{g(t)}(x, r).$
In particular,
\begin{align}\label{eqn: contain b}
	\vol_{g(t)}(B_{g(t)}(x,r)) & \geq \vol_{g(t)}(B_{g(r^2)}(x, (1-\e)r)).
\end{align}
Then, since $R\geq-\delta,$ by \eqref{eqn: volume forms}, we see that 
\begin{align}
	\vol_{g(t)}(B_{g(r^2)}(x, (1-\e)r))& \geq (1-\e)  \vol_{g(r^2)}(B_{g(r^2)}(x, (1-\e)r))\,.
\end{align}
We appeal to \eqref{eqn: under reg scale euclidean volumes} once more to conclude.
\end{proof}
  Wang shows in \cite[Theorem 5.9]{BingWangB} that under a lower bound on the Ricci curvature, the local versions of these two conditions are equivalent, and are also equivalent to an almost-Euclidean local isoperimetric constant. 

We will use the following elementary lemma that follows from integrating the scale-invariant curvature estimate \eqref{eqn: curvature estimates} along the Ricci flow; see \cite[Lemma 3.8]{LNN1}.
\begin{lemma}\label{lemma: simple}
	Fix $n\geq 2$  and $\beta\in(0,1/4)$. There exists $\lambda=\lambda(n,\beta)$ small enough such that if $(M,g(t))_{t \in (0,1]}$ is a Ricci flow satsifying \eqref{eqn: curvature estimates}, then
	for any $x \in M$  and $0<s_1 \leq s_2 \leq 1$, we have 		
\begin{equation}\label{eqn: bad bound general}
		\Big(\frac{s_1}{s_2}\Big)^\beta g(s_1) \leq g(s_2) \leq \Big(\frac{s_1}{s_2}\Big)^{-\beta} g(s_1).
	\end{equation}
	Consequently, for any $r>0$,  
	\[
	B_{g(s)}(x , r s^{1/2}) \subset B_{g(t) } (x, r t^{1/2}) \qquad \text{for all } 0 \leq s \leq t \leq 1\,.
	\]
	Moreover, for $t_0 \in (0,1)$, $r_0 \in (0,1)$  and $\e>0$ fixed, there is $\lambda$ small enough such that 
	\begin{equation}
		\label{eqn: simple 2}
B_{g(t_0)}(x, r_0) \subset B_{g(1)}\left(x, (1+\e) r_0\right).
		\end{equation}
\end{lemma}

Finally, we recall  the decomposition theorem from \cite[Theorem~5.1]{LNN1} for a 
Riemannian manifold with almost-Euclidean lower bounds on scalar curvature and Perelman entropy,
which decomposes a ball $B_{g(1)}(x, 8)$ into a countable union of sets $\{\mathcal{G}_k\}_{k \in \mathbb{N}}$ with volumes decaying geometrically in $k$ and on which the metric tensors $g(0)$ and $g(1)$ are pointwise comparable up to $(1\pm \e)^k$-multiplicative constants. 
\begin{lemma}[Decomposition theorem]
	\label{lemma: decomposition}
 For each $\e>0$ there exists $\delta= \delta(n,\e)>0$, such that the following holds. 
 Let $(M,g(t))_{t \in (0,1]}$ be a complete Ricci flow with bounded curvature with $g=g(0)$ satisfying \eqref{e: scalar and entropy lower bound}.
Fix $x_0 \in M$. The ball $B_{g(1)}(x_0,8)$ can be decomposed into good sets $\mathcal{G}^k$ and a bad set $\mathcal{A}$ in the following way:
 \begin{equation}\label{eqn: decomposition}
 	B_{g(1)}(x_0,8) = \bigcup_{k=1}^\infty \mathcal{G}^{k} \cup \mathcal{A}
 \end{equation} 	
 where
 \begin{enumerate}
 	\item\label{item: decomp A vol}  $\vol_{g(0)}(\mathcal{A}) = 0.$
 	\item\label{item: decomp G bounds} For all $x \in \mathcal{G}^{k}$ and for all $s,t\in (0,1]$, the metrics satisfy 
 	\begin{equation}
 		{(1-\e)^{k} g(s) \leq g( t) \leq (1+\e)^{k} g(s)}
	\end{equation}  
 	\item\label{item: decomp volume bounds G} For each $k\geq 2$, we have 
 {	$\vol_{g(0)}(\mathcal{G}^{k}) \leq (1+\e)^{k} \e^{k-1}.$}
 	\item\label{item: decomp Ak bounds}  
 	For each $k\in \mathbb{N}$, let $\mathcal{A}^k = B_{g(1)}(x_0, 8) \, \Big\backslash \,  \bigcup_{\ell=1}^k \mathcal{G}^\ell$ be the complement of the first $k$ good sets. There is a countable collection $\mathcal{C}^k$ and a mapping $y\mapsto t_y$ for $y \in \mathcal{C}^k$ such that 
 	\begin{equation}\label{eqn: content bound 1}
 	 	\mathcal{A}^k \subseteq \bigcup_{y\in\mathcal{C}^k	} B_{g(t_y)}(y, 12t_y^{1/2}) \qquad \text{with} \qquad \sum_{y\in\mathcal{C}^k} t_y^{n/2}\leq  \e^{k}.
 	\end{equation}
 \end{enumerate}

\end{lemma}

\section{Epsilon Regularity}
In this section, we prove Theorem~\ref{thm: GH}.
To begin,  we show that if $(M,g)$ satisfies the assumptions of Theorem~\ref{thm: GH} and $(M,g(t))_{t \in (0,1]}$ is the Ricci flow with $g(0)=g$ (whose existence is guaranteed by Lemma~\ref{lemma: initial RF}), then 
\begin{equation}
	\label{eqn: sec 3 intro ball}
B_{g(0)}(x,1-\e) \subset B_{g(1)}(x, 1)
\end{equation}
 provided $\delta$ is chosen sufficiently small. 
 Once this is shown, we can apply this and the opposite containment \eqref{eqn: one containmentA} and iterate at all points and scales to show \eqref{eqn: biholder ests th1}. Moreover, the containment \eqref{eqn: sec 3 intro ball} will  allow us to decompose the ball $B_{g(0)}(x, 1)$ according to Lemma~\ref{lemma: decomposition} to prove the $W^{1,p}$ estimates \eqref{eqn: w1p 1 intro}. While all of the results in section~\ref{sec: prelim} hold under only the lower bounds on scalar curvature and entropy, the proofs in this section necessarily make use of the upper bounds \eqref{en: vol upper bound theorem statement}
on the volumes of balls: examples in \cite{LNN1, LNNSurvey} show the containment \eqref{eqn: sec 3 intro ball} is false  in the absence of assumption \eqref{en: vol upper bound theorem statement}, even with $1-\e$ replaced by a tiny radius $c>0$.

The containment  \eqref{eqn: sec 3 intro ball} follows from the distance distortion estimates shown in \cite{BingWangB}, which in turn call on results from \cite{TianWang}. Here we give a slightly different proof, which combines two ingredients: lower density estimates and volume control. The lower density estimate, Proposition~\ref{prop: lower density estimate} below, says that if \eqref{eqn: sec 3 intro ball} fails to hold, i.e. $B_{g(0)}(x,1-\e)$ is {\it not} entirely contained in $B_{g(1)}(x,1)$, 
 then there is a $g(0)$ ball of definite size that lies inside the slightly bigger ball $B_{g(0)}(x,1-\e/2)$ and outside the slightly smaller ball $B_{g(1)}(x, 1-\e/2)$.
\begin{proposition} \label{prop: lower density estimate}
Fix $n\geq 2$ and $\e>0$. There exist $\delta = \delta(n,\e)>0$ such that the following holds. Let $(M,g)$ satisfy the hypotheses of Theorem~\ref{thm: GH}.
Then for any $x \in M$ and $y \in B_{g(0)}\left(x, 1-\e\right) \backslash B_{g(1)}\left(x, 1\right)$, we have 
\[
B_{g(0)}\big(y, \e/ 20\big) \subset B_{g(0)}\big(x, 1-\e/2\big) \big\backslash B_{g(1)}\big(x, 1-\e / 2\big).
\]
\end{proposition}
So, if \eqref{eqn: sec 3 intro ball} fails  for some $x$, then a  $g(0)$ ball of definite radius---thus of definite $g(0)$ volume by Lemma~\ref{lemma: volume noncollapsing}---is contained in a $g(0)$ ball of radius $\rho_0=1-\e/2$ and the complement of a $g(1)$ ball of the same radius and center.  This possibility is ruled out by the volume control of Proposition~\ref{prop: volume control}, which says that the set difference between a $g(0)$ ball and a $g(1)$ ball with the same radius and center must have very small volume for sufficiently small $\delta.$ This fact strongly uses the almost-Euclidean volume assumption \eqref{en: vol upper bound theorem statement}, as well as a lemma from \cite{LNN1} letting us compare $g(0)$ and $g(1)$ volumes of a fixed set.
\begin{proposition}\label{prop: volume control}
	Fix $n\geq 2$ and $\eta >0$.There exists $\delta = \delta (n,\eta)$ such that if $(M,g)$ satisfies the hypotheses of Theorem~\ref{thm: GH}, then for any $x \in M$ and $\rho_0 \in [1/5,1]$, 
	\begin{equation}
	\label{claim 1}
\vol_{g(0)} \big(B_{g(0)}(x, \rho_0) \setminus B_{g(1)}(x, \rho_0)\big)  \leq \eta.
\end{equation}
\end{proposition}
We prove Propositions~\ref{prop: lower density estimate} and \ref{prop: volume control} in sections~\ref{s3.1} and \ref{s3.2} respectively, and complete the proof of Theorem~\ref{thm: GH} in section~\ref{s3.3}.

\subsection{Lower density estimate}\label{s3.1}
To prove Proposition~\ref{prop: lower density estimate},  we first prove the rough containment $B_{g(0)}(x,1)\subset B_{g(1)}(x,5)$. Here, we crucially use the control \eqref{en: vol upper bound theorem statement} on volumes of balls, although in this step it is not essential that the contant is almost-Euclidean.  
The proof uses a known Ricci flow argument (see for instance  \cite{ChenWang,BZ1})  estimating the number of $g(t)$-balls needed to cover $g(0)$-geodesics, using the volume bounds \eqref{en: vol upper bound theorem statement} and noncollapsing of Lemma~\ref{lemma: volume noncollapsing}. 
\begin{lemma} 
\label{lem: ball 1} Fix $n \geq 2$.  There exist $\delta= \delta(n)>0$  such that the following holds. Let $(M,g)$ satisfy the hypotheses of Theorem~\ref{thm: GH} and let $(M, g(t))_{t \in (0,1]}$ be the Ricci flow with $g(0)=g$. 
For each  $ t\in (0,1]$ and $x, y \in M$ with $d_{g(0)}(x, y) \geq t^{1 / 2}$,
\begin{equation}
	\label{e: dd rough}
d_{g(t)}(x, y) \leq 4 d_{g(0)}(x, y) 
\end{equation}
and  
\begin{equation}\label{e: rough containment of balls}
B_{g(0)}\big(x, t^{1 / 2}\big) \subseteq B_{g(t)}\big(x, 5 t^{1 / 2}\big)\,.
\end{equation}
\end{lemma}

\begin{proof}
Fix $t \in (0,1]$ and fix $x,y\in M$ satisfying $d_{g(0)}(x, y) \geq t^{1 / 2}$, and set 
$d_0 = d_{g(0)}(x,y )$.\smallskip
	
\noindent{\it Case 1: $d_0 \leq 1/2$.} Let $\gamma$ be a minimizing geodesic from $x$ to $y$ with respect to the metric $g(0)$. Consider a maximal subset $\{y_i\}_{i=1}^N$ of $\text{Im}(\gamma)$ such that the balls $B_{g(t)}(y_i, d_0)$ are pairwise disjoint. In this way, the collection $\{ B_{g(t)}(y_i,2d_0)\}_{i=1}^N$ is a covering of $\text{Im}(\gamma)$ and thus
\begin{equation}\label{eqn: dt bound in terms of N}
	d_{g(t)}(x,y) \leq L_{g(t)}(\gamma) \leq 4d_0N\,.
\end{equation}	
Here $L_{g(t)}(\gamma)$ denotes the length of $\gamma$ with respect to the metric $g(t)$.
We will show that $N=1$ provided $\delta$ is chosen sufficiently small. Indeed, let $\e=\e(n)>0$ be a fixed small number to be specified below. Taking $\delta$, small enough depending on $\e$ and $n$,   we find from \eqref{eqn: one containmentA} that for each $i \in \{1,\dots, N\}$,
\begin{align}\label{eqn: containment}
	B_{g(t)}(y_i, d_0) \subseteq B_{g(0)}\big(y_i, (1+ \e)d_0\big) .
\end{align}
In particular, since the balls on the left-hand side of \eqref{eqn: containment} are pairwise disjoint,
\begin{equation}\label{eqn: vol bd 2}
	\sum_{i=1}^N \vol_{g(t)}(B_{g(t)}(y_i, d_0)) \leq \vol_{g(t)}\big(B_{g(0)}(x,(1+\e)d_0)\big).
\end{equation}
The lower bound on scalar curvature and \eqref{eqn: volume forms} ensure that 
$ d\vol_{g(t)}\leq (1+\e) d\vol_{g(0)}$ for all $t \in (0,1]$, provided we choose $\delta>0$ small enough in terms of $\e$.
 Therefore, keeping in mind that $d_0\leq 1/2$ so $(1+\e)d_0<1$, we can apply the volume growth assumption \eqref{en: vol upper bound theorem statement} to bound right-hand side of \eqref{eqn: vol bd 2} above:
 \begin{equation}\label{eqn: vol bd 3}
 \begin{split}
 	\vol_{g(t)} \big(B_{g(0)}(x,(1+\e) d_0)\big)
 	& \leq (1+\e) \vol_{g(0)}\big(B_{g(0)}(x,(1+ \e)d_0)\big)\\ & \leq \om_n(1+ \e)^{n+1} d_0^n.
 \end{split}
 \end{equation}
On the other hand, up to further decreasing  $\delta$ depending on $n$ and $\e$,  Lemma~\ref{lemma: volume noncollapsing} tells us the balls $B_{g(t)}(y_i, d_0)$ on the left-hand side of \eqref{eqn: vol bd 2} each have $g(t)$-volume at least $(1-\e) \om_n d_0^n$. 
Using this and \eqref{eqn: vol bd 3} to bound the left- and right-hand sides of \eqref{eqn: vol bd 2} respectively, we have
\begin{align*}
	 N (1-\e) \om_n  d_0^n \leq \om_n (1+\e)^{n+1}d_0^n.  
\end{align*} 
Dividing through by $\omega_n d_0^n$ and taking $\e>0$ small enough depending on $n$, we determine that $N<2.$ So, recalling \eqref{eqn: dt bound in terms of N} and the definition of $d_0$, we conclude that \eqref{e: dd rough} holds in this case.\smallskip

\noindent{\it Case 2: $d_0 \geq 1/2$.} Let $\gamma:[0,d_0] \to M$ be a minimizing geodesic from $x$ to $y$ with respect to the metric $g(0)$ parameterized by arclength. Let 
$
	\rho = \frac{d_0}{\lceil 2d_0 \rceil }\,.
$
Note that $\rho\in [1/4, 1/2]$, so in particular $4\rho \geq t^{1/2}$ for any $t \in (0,1]$.
For $i = 0,\dots, \lceil 2d_0\rceil$, set $x_i =\gamma(i\rho)$. By applying Case 1 to the $g(0)$-geodesic segments from $x_{i-1}$ to $x_i$, we find that
\begin{align}
	d  = \sum_{i=1}^{\lceil 2d_0 \rceil} d_{g(0)}(x_{i-1}, x_i) & \geq \frac{1}{4} \sum_{i=1}^{\lceil 2d_0 \rceil} d_{g(t)}(x_{i-1}, x_i)   \geq \frac{1}{4} d_{g(t)}(x,y)\,.
\end{align}
This completes the proof of \eqref{e: dd rough}.\smallskip

To see how the containment of balls \eqref{e: rough containment of balls} follows, fix any $y \in B_{g(0)}(x,t^{1/2})$. Again letting $d_0 = d_{g(0)}(x,y)$ (so that now $d_0 < t^{1/2}$), we apply \eqref{e: dd rough} at the time slice $d_0^2$ to see that $y \in B_{g(d_0^2)}(x, 4(1+\e) d_0)$. Next, by Lemma~\ref{lemma: simple},  we have 
	\begin{equation}\label{eqn: bad containment of balls}
		B_{g(d_0^2)} \big(x, 4 (1+\e) d_0\big) \subseteq B_{g(t)}\big(x,4 (1+\e)^2t^{1/2} \big)
	\end{equation}
	since $d_0^2 \leq t$. Choosing $\e>0$ small enough so that $4(1+\e)^2\leq 5$ completes the proof.
\end{proof}
Now we  use Lemma~\ref{lem: ball 1} to prove the lower density estimate.
\begin{proof}[Proof of Proposition~\ref{prop: lower density estimate}] 
Suppose that we may find some $y$ as in the statement of the proposition, i.e. such that $d_{g(0)}(y, x) \leq 1-\e$ and $d_{g(1)}(y, x)>1$. We claim that 
	\begin{align}
\label{c1}		B_{g(0)}(y ,\e / 20) &\subset B_{g(0)}(x, 1-\e / 2) \\
\label{c2}	  \text{ and}	\quad B_{g(0)}(y ,\e / 20) &\subset M\setminus B_{g(1)}(x, 1-\e / 2)\,.
	\end{align}
The first containment \eqref{c1} is immediate from the triangle inequality:  for $z \in B_{g(0)}(y, \e / 20)$, 
\[
d_{g(0)}(z, y) \leq d_{g(0)}(z, y)+d_{g(0)}(y, x)  \leq \e /20+(1-\e) \leq 1-\e / 2 .
\]
Toward showing \eqref{c2}, we claim that for $\delta = \delta(n,\e)$ sufficiently small, we have  
\begin{equation}
	\label{eqn: small ball bad containment}
	B_{g(0)}(x, \e/20) \subseteq B_{g(1)}(x,\e/2) 
\end{equation}
for each $x \in M.$  Indeed,  let $t_0 = \e^2/400$. 
	Applying  Lemma~\ref{lem: ball 1} at scale $t_0$ tells us that $B_{g(0)}(x, \e/20) \subset B_{g(t_0)} (x, \e/4).$ 
	Next, by  Lemma~\ref{lemma: simple}, we can choose $\lambda$ and thus $\delta$ sufficiently small depending on $t_0$ (thus $\e$) and $n$ to find that $B_{g(t_0)} (x, \e/4) \subset B_{g(1)}(x, \e/2)$ for every $x \in M$. This yields \eqref{eqn: small ball bad containment}.
	
So, thanks to \eqref{eqn: small ball bad containment},  to show \eqref{c2} it suffices to show that 
$ B_{g(1)}(y, \e / 2) \subset M \setminus B_{g(1)}(x, 1-\e / 2).$ Again, this just follows from the triangle inequality: for $z \in B_{g(1)}({y}, \e / 2)$, we have
\[
d_{g(1)}(z, x) \geq d_{g(1)}(x, y)-d_{g(1)}(y, z) \geq 1-d_{g(1)}(z,y) \geq 1-\e/2.
\]
This completes the proof.
\end{proof}

\subsection{Volume control}\label{s3.2}
Using only the first iteration of the decomposition theorem, Lemma~\ref{lemma: decomposition}, i.e. just splitting $B_{g(1)}(x, 8)$ into the first ``good set'' $\mathcal{G}_1$ and its complement of small volume, we have the following volume control lemma.
 \begin{lemma}
 	\label{prop: volumes} 
 Fix $n\geq 2$ and  $\eta>0$.  There exists $\delta = \delta(n, \eta)$ such that, for any $\Omega\subset M$ with $\text{diam}_{g(1)}(\Omega) \leq 8$,  
	 \[
	 \left|  \vol_{g(0)}(\Omega) -  \vol_{g(1)} (\Omega )\right|  \leq \eta \,.
	 \]
 \end{lemma}
This lemma is one of the tools in the proof of Proposition~\ref{prop: volume control}. 
\begin{proof}[Proof of Proposition~\ref{prop: volume control}]
Let $V_0 = 
\vol_{g(0)} \big(B_{g(0)}(x, \rho_0) \setminus B_{g(1)}(x, \rho_0)\big)$, so we aim to show that $V_0   \leq\eta.$
 To this end, 
	we first  slightly enlarge the $g(0)$ ball appearing on the left-hand side of \eqref{claim 1}:  thanks to the containment  \eqref{eqn: one containmentA} and the assumption $\rho_0\geq 1/5$, we can take $\delta=\delta(n,\eta)$ sufficiently small so that $ B_{g(1)}(x, \rho_0) \subset B_{g(0)}(x, (1+\eta)\rho_0)$. So, letting $V_0 $ denote the left-hand side of \eqref{claim 1},  we have 
\[
\begin{aligned}
V_0 & \leq\vol_{g(0)} \big(B_{g(0)}(x, (1+\eta) \rho_0) \setminus B_{g(1)}(x, \rho_0)\big)  \\
& = \vol_{g(0)}\big(B_{g(0)}(x, (1+\eta)\rho_0)\big)-\vol_{g(0)}\big(B_{g(1)}(x, \rho_0)\big).
\end{aligned}
\]
We bound the first term on the right-hand side using the assumption \eqref{en: vol upper bound theorem statement} of an almost-Euclidean upper bound for volumes of $g(0)$ balls:
\[
 \vol_{g(0)}\big(B_{g(0)}(x, (1+\eta)\rho_0)\big ) \leq (1+\delta)(1+\eta)^n \omega_n\rho_0^n.
 \]
To bound the second term below, we apply Proposition~\ref{prop: volumes}, allowing us to compare $g(0)$ and $g(1)$ volumes, followed by the volume noncollapsing of Lemma~\ref{lemma: volume noncollapsing}:
\[
\begin{aligned}
\vol_{g(0)}\big(B_{g(1)}(x, \rho_0)\big)
& \geq (1-\eta ) \vol_{g(1)}(B_{g(1)}(x, \rho_0)) \geq  (1-2 \eta) \omega_n\rho_0^n.
\end{aligned}
\]
Putting these facts together we see that 
\begin{align*}	
V_0 &  \leq \omega_n \rho_0^n \big((1+ \delta) (1+\eta)^n - (1-2 \eta) \big)  \\
& \leq  \omega_n \rho_0^n C_n \eta  = C_n\eta. 
\end{align*}
Here we have assumed without loss of generality that $\delta\leq \eta$ and have used that $\rho_0\leq 1.$ Repeating this argument with $\eta' =\eta/C_n $  proves the proposition.
\end{proof}
\subsection{Conclusion}\label{s3.3}
Now we prove Theorem~\ref{thm: GH}.
\begin{proof}[Proof of Theorem~\ref{thm: GH}]
Let $(M, g(t))_{t \in (0,1]}$ be the Ricci flow with $g(0)=g$, whose existence is guaranteed by Lemma~\ref{lemma: initial RF}. Fix $x \in M$.

\noindent{\it Step 1:}
First, we show \eqref{eqn: sec 3 intro ball}, that is, we claim that
\[
B_{g(0)}(x, 1-\e) \subset  B_{g(1)}(x, 1),
\]
for $\delta$ sufficiently small depending on $n$ and $\e$. Let $\eta>0$ be a fixed number depending on $n$ and $\e$ to be specified below. Choose $\delta = \delta(n,\e,\eta)$ small enough according to Propositions~\ref{prop: lower density estimate} and \ref{prop: volume control} and Lemma~\ref{lem: ball 1}.
Set $\rho_0 = 1- \e/2$.  
If $\rho_0 <1/5,$ the claim follows from Lemma~\ref{lem: ball 1}, so we assume that $\rho_0 \geq 1/5$. 
Suppose there is some 
point $y \in B_{g(0)}(x, 1-\e)$ with $y \not\in B_{g(1)}(x, 1).$ The lower density estimate of Proposition~\ref{prop: lower density estimate} then implies that
\[
B_{g(0)}(y, \e / 20) \subset B_{g(0)}(x, \rho_0) \setminus B_{g(1)}(x,\rho_0),
\]
and thus Proposition~\ref{prop: volume control} implies that
\begin{equation}
\vol_{g(0)}\big(B_{g(0)}(y, \e / 20 )\big)  \leq 100 \eta
\end{equation}
by containment. On the other hand, by Lemma~\ref{lemma: volume noncollapsing}, the right-hand side of this expression is bounded below by $\om_n \e^n/21^n$ for $\delta=\delta(n,\e)>0$ small enough. Choosing $\eta \leq \om_n \e^n/21^n$, we see that such a point $y$ cannot exist and thus the claim holds. Since all hypotheses are preserved under rescaling the metric $g \mapsto t^{-1}g$ for $t \in (0,1]$, we note that this claim shows that 
\[
B_{g(0)}(x, (1-\e) t^{1/2} ) \subset B_{g(t)}(x, t^{1/2})
\]
for all $t \in (0,1]$ and $x \in M$; up to further decreasing $\delta$ we may assume this holds for all $t\in (0,2]$.
\smallskip

\noindent{\it Step 2:} 
Together step 1 and \eqref{eqn: one containmentA} tell us that, up to further decreasing $\delta>0$, we have 
\begin{equation}
	\label{eqn: dd}
B_{g(t)}\big(x, (1-\e) t^{1/2} \big) \subseteq B_{g(0)}(x, t^{1/2}) \subseteq B_{g(t)}\big(x,(1+\e) t^{1/2}\big)
\end{equation}
for all $x\in M$ and $t \in (0,2]$. 
By Lemma~\ref{lemma: initial RF}, we have a smooth diffeomorphism $\psi:B_{g(1)}(x,16)\to \Omega\subset \R^n$, with inverse $\phi = \psi^{-1}$ such that $\psi(x)=0$ and
\begin{equation}
			\label{eqn: metric bounds 2}
		(1-\e) g_{euc}  \leq   \phi^*g(1)\leq (1+\e)g_{euc} 
\end{equation}
for all $x\in \Omega$,  as long as $\ETA$ (and hence $\delta)$ has been chosen to be sufficiently small depending on $\e$ and $n$. Set $U= \psi(B_{g(0)}(x,1))$. By \eqref{eqn: dd} with $t=1$ and \eqref{eqn: metric bounds 2}, we have 
\begin{equation}\label{eqn: U containment}
	B(0,1-\e)\subseteq U\subseteq B(0,1+\e).
\end{equation}

Let us establish the bi-H\"{o}lder estimates \eqref{eqn: biholder ests th1}, which is standard from \eqref{eqn: dd}. Thanks to \eqref{eqn: metric bounds 2}, it suffices to show that the identity map is a bi-H\"{o}lder between $g$ and $g(1)$ satisfying
\begin{equation}\label{eqn: biholder identity}
 	(1-\e)d_{g}(y,z)^{1/\alpha}\leq d_{g(1)}(y,z) \leq (1+\e)d_{g}(y,z)^\alpha
\end{equation}
for all $y,z \in B_{g}(x,1)$. To this end, fix any such $y,z$ and let $d=d_{g}(y,z)\leq 2.$ By \eqref{eqn: dd} at scale $t =d^2,$ we have
\begin{equation}\label{eqn: d1}
(1-\e)	d\leq d_{g(d^2)}(x,y) \leq (1+\e)d.
\end{equation}
Let $2\beta = 1-\alpha$. Up to further decreasing $\ETA$ (and thus $\delta$) to comply with  Lemma~\ref{lemma: simple}, we see that \eqref{eqn: bad bound general} yields
\begin{equation}\label{eqn: d2}
(1-\e)d^{2\beta} d_{g(d^2)}(x,y) 	\leq d_{g(1)}(x,y) \leq (1+\e)d^{-2\beta} d_{g(d^2)}(x,y).
\end{equation}
Combining these two estimates \eqref{eqn: d1} and \eqref{eqn: d2} establishes \eqref{eqn: biholder identity}. We have thus shown \eqref{eqn: biholder ests th1}.

Finally, we can repeat the proof of theorem 1.11 in \cite[Section 6]{LNN1} identically, with the key improvement that the decomposition of Lemma~\ref{lemma: decomposition} now holds on $B_g(0,1)$ thanks to step 1, to conclude the $W^{1,p}$ estimates \eqref{eqn: w1p 1 intro}.
\end{proof}

\section{Limit spaces}
In this section we establish Theorem~\ref{thm: limit structure theorem}. We first prove the compactness and properties (1) and (2) of Theorem~\ref{thm: limit structure theorem} in section~\ref{ss:4.1}, then introduce two additional lemmas in section~\ref{ss:4.2} before proving properties (3) and (4) in section~\ref{ss:4.3}.

\subsection{Compactness, topological structure, and measure of balls.}\label{ss:4.1} 
We start by proving that sequences of pointed Riemannian manifolds as in the statement of Theorem~\ref{thm: limit structure theorem} have pointed measured Gromov-Hausdorff limits, that the limit space $X$ is a topological manifold, and that the measure of balls up to scale one in the limit space are almost equal to the volume of Euclidean balls of the same radius.
\begin{proof}[Proof of Theorem~\ref{thm: limit structure theorem}] Take $\delta=\delta(n, \e)>0$ as in Theorem~\ref{thm: GH}. Up to replacing each $g_i$ by the rescaled metric $\rho_0^{-2} g_i$ where $\rho_0^2 = \min\{\delta/C_0,\tau_0/2, r_0^2/4\}$, we may assume that $C_0=\delta$, $\tau_0=2$, and $r_0 =2$  in \eqref{eqn: hp limit}. In this way, each $(M_i,g_i,x_i)$ satisfies the hypotheses of Theorem~\ref{thm: GH}. At various points in the proof, we will pass to subsequences, which we will not relabel.

{\it Step 0.}  We first show that the sequence $\{(M_i, g_i, x_i, d\vol_{g_i})\}$ has a convergent subsequence in the pointed measured Gromov-Hausdorff topology. At a fixed $y\in M_i,$ we let $\text{Cov}_R(r) $ denote the minimum number of balls of radius $r$ needed to cover $B_{g_i}(y,R)$. As a direct consequence of Theorem~\ref{thm: GH}, we find that $\text{Cov}_1(r)$ is bounded above by a function $N_1(r)=N_1(r,n,\e,\delta)$. Furthermore, we have assumed that $\vol_{g_i}(B_{g_i}(y,r))$ is bounded above by a function $V(r)=V(r,n,\e,\delta)$ for each $r\leq 1$. As this holds at every point, a simple induction argument then establishes that $\text{Cov}_R(r)$ is bounded above by a function $N_R(r)=N_R(r,n,\e,\delta)$ for $R\geq 1$ as well, and consequently, $\vol_{g_i}B_{g_i}(y,R)$ is bounded above by a function $V(R)=V(R,n,\e,\delta)$ for all $R\geq 1$ as well.  Hence, the sequence $\{(M_i,g_i,x_i,d\vol_{g_i})\}$ is precompact in the pointed measured Gromov-Hausdorff sense and a subsequence converges to a proper  pointed metric measure space $(X,d,x_\infty,m)$; see for instance \cite[Theorem 11.4.7]{MMSBOOK}.

{\it Step 1.} Next we establish the manifold structure of $X$ by constructing an atlas of charts, proving part (1) of the theorem. 
 We follow Petersen's presentation of Cheeger's fundamental theorem of convergence theory, \cite[Chapter 10, Theorem 72]{petersenBook}. 
 For each fixed $i$, Theorem~\ref{thm: GH} establishes the existence of an atlas of charts $\phi_{i\ell}:B_{g_{euc}}(0,1)\to U_{i\ell}\subset M_i$, where each $\phi_{i\ell}$ is a bi-H\"{o}lder homeomorphism with uniform bounds on the bi-H\"{o}lder norms. Without loss of generality, we may assume that the index set $\{\ell\}$ is the same for all $i$, and that we have indexed the charts in such a way that $x_i\in U_{i1}, $ and that $B(p,R)$ is covered by the first $N_R(1)$ charts. For each $\ell$, the uniform bi-H\"{o}lder estimates ensure that, up to a subsequence, $\{\phi_{i\ell}\}$ converges to a bi-H\"{o}lder map  $\phi_\ell: B_{g_{euc}}(0,1)\to U_\ell\subset X$. In particular, $\phi_\ell$ is a homeomorphism. Up to selecting a diagonal subsequence, this convergence occurs for all $\ell\in \mathbb{N}.$ Finally, it is easy to check from the Gromov-Hausdorff convergence that every $x\in X$ is contained in $U_\ell$ for some $\ell$. Therefore, the collection of maps $\{\phi_\ell\}$ provides an atlas of topological charts for $X.$

{\it Step 2.} We now show the bound \eqref{eqn: limit space volumes of balls} for the $m$ measure of balls to establish part (2) of the theorem. The initial rescaling makes $r_0=1$, and up to further rescaling, it suffices to take $r=1$ in \eqref{eqn: limit space volumes of balls}.  Any ball $B_{\infty}:=B_d(x,1)\subset X$ is the Gromov-Hausdorff limit of balls $B_i:=B_{g_i}(x_i,1)\subset (M_i,g_i)$, and as in the previous step comes equipped with a bi-H\"{o}lder map $\phi: B_{g_{euc}}(0,1)\to B_\infty$ arising as the limit of the maps $\phi_i:B_{g_{euc}}(0,1)\to B_i$ of Theorem~\ref{thm: GH}. We denote by $\psi_i$ and $\psi$ the inverses of $\phi_i$ and $\phi$ respectively. We claim that 
 \begin{equation}\label{eqn: limc}
 	m(B_\infty) = \lim_{i\to \infty}\vol_{g_i}(B_i).
 \end{equation}
Then \eqref{eqn: limit space volumes of balls} follows immediately from \eqref{eqn: limc}, because assumption \eqref{eqn: hp limit} and Lemma~\ref{lemma: volume noncollapsing} ensure that $(1-\e) \omega_n \leq \vol_{g_i}(B_i) \leq (1+\delta) \om_n $.
 To prove \eqref{eqn: limc}, note that (up to a subsequence) the map $F_i= \phi\,\circ\psi_i:B_i\to B_\infty$ is a $1/i$ Gromov-Hausdorff approximation. 
 So, for a fixed function $f \in C_c(B_\infty)$,  we have $\lim_{i\to\infty}\int_{M_i} f_j\circ F_i\,d\vol_{g_i} \to  \int_X f_j\,dm$ from the weak convergence of measures.
 Take  a sequence of functions  $f_j \in C_c(B_\infty)$  converging in $L^1(X;m)$ to the characteristic function of $B_\infty.$ We have 
 \begin{equation}\label{eqn: lima}
 	\lim_{j\to\infty}\, \lim_{i\to\infty}\int_{M_i} f_j\circ F_i\,d\vol_{g_i}=\lim_{j\to\infty} \int_X f_j\,dm=m(B_\infty).
 \end{equation} 
 Now, we claim that the integrals on the left-hand side are uniformly bounded with respect to $i$ and $j$, so we may exchange the limits. We can assume that $f_j \leq \chi_{B_\infty}$, so, letting $\nu_i := (\psi_i)_\sharp m_i$ for $dm_i = d\vol_{g_i}$ be a sequence of measures defined on $B_{g_{euc}}(0, 1)$, it suffices to show that $\nu_i(B_{g_{euc}}(0,1))$ is uniformly bounded in $i$.  This fact then follows directly from Lemma~\ref{lemma: initial RF} and the volume control of Lemma~\ref{prop: volumes}. So, we may exchange order of the limits with respect to $i$ and $j$  in \eqref{eqn: lima}.
Next, applying the dominated convergence theorem and recalling the definition of $F_i$, we  find that 
 \begin{equation}\label{eqn: limb}
 	\begin{split}
 	\lim_{i\to\infty}\, \lim_{j\to\infty}\int_{M_i} f_j\circ F_i \,d\vol_{g_i}&= \lim_{i\to\infty} \vol_{g_i}(F_i^{-1}(B_\infty))  =\lim_{i\to\infty} \vol_{g_i}(B_i) .	
 	\end{split}
 \end{equation}
Together \eqref{eqn: lima} and \eqref{eqn: limb} imply \eqref{eqn: limc}.
We have thus proven the pointed measured Gromov-Hausdorff  compactness part of Theorem~\ref{thm: limit structure theorem} and the properties (1) and (2) of the limiting pointed metric measure spaces.
 \end{proof}

\subsection{Intermediate lemmas}\label{ss:4.2}
Toward proving the remaining two properties (3) and (4) of limit spaces in Theorem~\ref{thm: limit structure theorem}, we first prove two intermediate lemmas. 
The first lemma will directly lead to property (3):
\begin{lemma}\label{lem: measures and densities} Fix $\e>0$, $\e'>0$ and $r_0>0$. 
	Let $(X,d)$ be a locally compact separable metric space, and let $m$ be a Radon measure on $X$ satisfying 
	\begin{equation}\label{eqn: lem volumes close}
		(1-\e)\om_n r^n \leq m(B_d(x,r)) \leq (1+\e')\om_n r^n
	\end{equation}
	for every $B_d(x,r)\subset B_d(x_0,r_0)$. Then $m$ and $\Hi^n$ are mutually absolutely continuous in $B_d(x_0,r_0)$, and $m=f\Hi^n $ with $1-\e\leq f \leq 1+\e'$.
\end{lemma}
\begin{proof} We show that for any $m$-measurable set $\Omega\subseteq B_d(x_0,r_0)$,  we have 
\begin{equation}\label{eqn: m, haus}
	(1-\e) \Hi^n(\Omega)\leq m(\Omega) \leq (1+\e')\Hi^n(\Om).
\end{equation}
First we prove the upper bound in \eqref{eqn: m, haus}. Since $m$ is a Radon measure, it suffices to assume that $\Omega$ is compact, and hence $\Om \subseteq B_d(x_0,r_0-\delta_0)$ for some $\delta_0>0$. Fix $\delta<\delta_0$. From the definition of Hausdorff measure, we may find a covering $\{\Om_\ell\}_{\ell=1}^\infty$ of $\Omega$ with $r_\ell:=\text{diam}(\Om_\ell)/2  \leq \delta$ and such that
\begin{equation}
	\Hi^n(\Omega) \geq \om_n\sum_{\ell=1}^\infty r_\ell^n -\delta.
\end{equation}
 For each $\ell$, the set $\Om_\ell$ is contained in $B_d(x_\ell,r_\ell)$ for some $x_\ell\in B(x_0,r_0)$, and thus $\{B_d(x_\ell,r_\ell)\}$ is a covering of $\Omega$ as well. From the second inequality  in \eqref{eqn: lem volumes close}, we thus have
 \begin{equation}
 	\om_n\sum_{\ell=1}^\infty r_\ell^n -\delta \geq \frac{1}{1+\e'}\sum_{\ell=1}^\infty m(B_d(x_\ell,r_\ell)) -\delta \geq \frac{1}{1+\e'}\, m(\Omega)-\delta.
 \end{equation}
 Taking $\delta\to 0$ proves the upper bound in \eqref{eqn: m, haus}. 
 
 Now we show the lower bound in \eqref{eqn: m, haus}. As $m$ is a Radon measure, it suffices to consider $\Omega$ open. Fix $\delta>0$ and let $\mathcal{B} = \{ B_d(x,r) :  B_d(x,r) \subset \Omega, r\leq 2\delta\}$, which is a covering of $\Omega$. We have assumed $X$ is locally compact and separable and $m$ is doubling in $B_d(x_0,r_0)$ by \eqref{eqn: lem volumes close}. So, by \cite[Remark 4.5(1)]{simonGMT},  $B_d(x_0, r_0)$ has the symmetric Vitali property with respect to $m$, in other words  we may find a countable pairwise disjoint subcollection $\mathcal{B}'= \{ B_d(x_\ell,r_\ell)\} \subset \mathcal{B}$ covering $\Omega$ up to an $m$-negligible set. Therefore, applying the first inequality in \eqref{eqn: lem volumes close}, we have
 \begin{equation}
 	m(\Omega)= \sum_{\ell=1}^\infty m( B_d(x_\ell,r_\ell) ) \geq (1-\e)\om_n \sum_{\ell=1}^\infty r_\ell^n \geq (1-\e)\Hi^n_\delta(\Omega).
 \end{equation}
Taking $\delta\to 0$ concludes the proof of the lower bound in \eqref{eqn: m, haus} and thus of the lemma.
\end{proof}

In the next lemma, we let $(M,g(t))_{t \in (0,1]}$ be the Ricci flow starting from $g$, whose existence is guaranteed by Lemma~\ref{lemma: initial RF}. Thanks to (the proof of) Theorem~\ref{thm: GH}, we know in particular that $B_g(x,2) \subset B_{g(1)}(x,8)$ for any $x \in M$ and so the decomposition of Lemma~\ref{lemma: decomposition} applies to $B_g(x,2).$ We use the notation $\mathcal{G}^k$ and $\mathcal{A}^k$ for the $k$th good and bad sets in Lemma~\ref{lemma: decomposition}, and can always assume they are intersected with $B_g(x,1)$.

 The idea of Lemma~\ref{rmk: content bound replacement} below is that, when restricted to the set $\mathcal{G}^k$ in Lemma~\ref{lemma: decomposition}, the identity map from $(B_g(x,1), g(0))$ to $B_g(x,1), g(1))$ is  $(1+\e)^k$ bi-Lipchitz. In turn, this implies that the map $\psi$ in Theorem~\ref{thm: GH} is a $(1+2\e)^k$ bi-Lipschitz map when restricted to the set $\mathcal{F}_k  = \cup_{\ell=1}^k \mathcal{G}^{\ell}$ by Lemma~\ref{lemma: initial RF}. In fact, in order to pass these Lipschitz maps to the limit toward proving the rectifiability of limit spaces, we need to extend this bi-Lipschitz property to some of the points outside of $\mathcal{G}^k$. To this end, and using the notation of Lemma~\ref{lemma: decomposition}, let 
		\begin{equation}\label{eqn: r_k def}
		r_k(x)=\begin{cases} 0 & \text{ if } x \in \mathcal{F}_k\\
			\max\{ t_y^{1/2} : y \in \mathcal{C}^k \text{ and } x \in B_{g(t_y)}(y, 12t_y^{1/2}) \}& \text{ else}.
		\end{cases} 
		\end{equation} 
\begin{lemma}
\label{rmk: content bound replacement}
Fix $n \geq 2$ and $\e>0$. There exists $\delta=\delta(n,\e )>0$ such that the following holds. Fix $(M,g)$ satisfying the assumptions of Theorem~\ref{thm: GH} and  $x_0\in M$. Fix $k \in \mathbb{N}$ and let  $r_k$ be as defined in \eqref{eqn: r_k def}. For any pair of points $x_1, x_2 \in B_g(x_0,1)$ such that  $ d_{g(0)}(x_1, x_2) \geq \e \max\{r_k(x_1),r_k(x_2)\}$, we have
 \begin{equation}\label{eqn: lip large scales 1}
	(1-\e)^k d_{g(0)}(x_1,x_2) \leq d_{g(1)}(x_1,x_2) \leq (1+\e)^k d_{g(0)}(x_1,x_2).
	\end{equation}
\end{lemma}

\begin{proof}
Choose $\delta=\delta(n, \e) >0$ small enough to apply Theorem~\ref{thm: GH} (for any choice of $\alpha, p$) and Lemma~\ref{lemma: decomposition}.
Let $d = d_{g(0)}(x_1, x_2)$. 
	We prove the case when $k=1$; the case $k\ge 1$ then follows by induction as in the proof of Lemma~\ref{lemma: decomposition} (see \cite[Theorem 5.1]{LNN1}). 
	First, we apply the distance distortion estimates \eqref{eqn: dd} at scale $d^2$ with $\e/2$  in place of $\e$ (by further decreasing $\delta$) to see that 
	\begin{equation}\label{eqn: compare 1}
		(1-\e/2)d \leq d_{g(d^2)}(x_1,x_2)\leq (1+\e/2)d .
	\end{equation}
Now, if $r_k(x_1) = r_k(x_2) = 0$, let $\bar y = x_1$. Otherwise, suppose without loss of generality that $\bar{r}: = r_k(x_1) \geq r_k(x_2)$ and let $\bar y\in \mathcal{C}^k$ be such that $r_k(x_1)=t_{\bar y}^{1/2}.$ We  claim that in either case,
	\begin{equation}\label{eqn: x1, x2 containment}
		x_1,x_2 \in B_{g(d^2)}(\bar y, 15d).
	\end{equation}
	Once \eqref{eqn: x1, x2 containment} is shown, then \cite[Proposition~5.6(2)]{LNN1} ensures that 
\begin{equation}\label{eqn: compare 2}
	(1-\e/2) d_{g(d^2)}(x_1,x_2) \leq d_{g(1)}(x_1,x_2)
	\leq (1+\e/2) d_{g(d^2)}(x_1,x_2);
\end{equation}
this estimate  essentially comes from the way the decomposition is defined and parabolic estimates for the Ricci flow, and for the case $k>1$ the multiplicative factors $1 \pm \e/2$ become $(1\pm\e/2)^k$. Once this is shown, together \eqref{eqn: compare 1} and \eqref{eqn: compare 2} establish \eqref{eqn: lip large scales 1}.

So, it remains to show the containment \eqref{eqn: x1, x2 containment}. In the case when $\bar y=x_1$ this is immediate for $x_1$ and follows from \eqref{eqn: compare 2} for $x_2$ (with $1+\e/2$ in place of $15$). We focus on the second case where $x_1 \in B_{g(\bar{r}^2)}( \bar y , 12 \bar r).$
 In this case, we see as a direct consequence of \eqref{eqn: containment of balls} and Lemma~\ref{lemma: simple} that $x_1 \in B_{g(d^2)}(\bar y, (1+2\e)12 d)$.  So, the containment \eqref{eqn: x1, x2 containment} holds for $x_1$ in this case. For $x_2,$ we have from \eqref{eqn: compare 1} that 
	\begin{align*}
	d_{g(d^2)}(x_2,\bar y) & \leq d_{g(d^2)}(x_1, \bar y) + d_{g(d^2)}(x_1,x_2)\\
	& \leq 	 (1+2\e) 12 d + (1+\e/2)d \leq 15 d.
	\end{align*}
This completes the proof. 
\end{proof}
\begin{remark}\label{rmk: A}
	{\rm 
Note  also that if we let $\psi$ be the map from Theorem~\ref{thm: GH}, then Lemma~\ref{rmk: content bound replacement} and Lemma~\ref{lemma: initial RF} together imply that
	\begin{equation}\label{eqn: lip large scales 2}
		(1-\e)^{k+1} d_{g}(x_1,x_2) \leq |\psi(x_1)-\psi(x_2)|\leq 
		(1+\e)^{k+1} d_{g}(x_1,x_2)
	\end{equation}
for any pair of points as in Lemma~\ref{rmk: content bound replacement}. 
}
\end{remark}
\subsection{Mutual absolute continuity and rectifiability}\label{ss:4.3}
We now complete the proof of Theorem~\ref{thm: limit structure theorem} by showing properties (3) and (4) for limit spaces.

\begin{proof}[Proof of Theorem~\ref{thm: limit structure theorem}, continued]
We continue with the same notation used in the first part of the proof of Theorem~\ref{thm: limit structure theorem} above. The limit pointed metric measure space $(X, m,x_\infty, d)$ is proper (e.g. by \cite[Theorem 11.4.7]{MMSBOOK}), thus in particular locally compact and separable. Moreover, by \cite[Corollary 3.3.47]{MMSBOOK}, the limit measure $m$ is a Radon measure. So, thanks to \eqref{eqn: limit space volumes of balls}, we are in a position to apply  Lemma~\ref{lem: measures and densities} with $\e'=\delta$ to any unit ball $B_\infty$ in $(X,d, x_\infty, m)$. Lemma~\ref{lem: measures and densities} shows that $m$ and $\mathcal{H}^{n}$  are mutually absolutely continuous (in every ball unit $B_\infty$ and thus globally), with $m =f\mathcal{H}^{n}$ satisfying $(1-\e)\leq f \leq (1+\delta)$. This proves (3).

Next,  we show that $X$ is $m$-rectifiable and $\mathcal{H}^{n}$-rectifiable.
As rectifiability is a local property,  it suffices to show that  $B_\infty $ is $m$- and $\Hi^n$-rectifiable. For $i$ fixed, let $B_i=\cup_{k=1}^\infty \mathcal{G}_i^k\cup\mathcal{A}_i$ be the decomposition given by Lemma~\ref{lemma: decomposition} applied to $B_i \subset B_{g_i(1)}(x,8)\subset M_i$, and let $\F_i^k:=\bigcup_{\ell=1}^k\mathcal{G}_i^\ell$.
 By Lemma~\ref{lemma: decomposition}, for fixed $k$, we have
\begin{equation}
	B_i \setminus \F_i^k \,\subset \,\bigcup_{a=1}^\infty B_{g_i}(y_{a,i},15t_{a,i}^{1/2}), \qquad \sum_{a=1}^\infty t_{a,i}^{n/2}\leq \e^k.
\end{equation}
Note that on the right-hand side of the inclusion, we have replaced balls with respect to $g_i(t_{ai}^{1/2})$ in Lemma~\ref{lemma: decomposition}(4) by balls of slightly larger radii with respect to $g_i=g_i(0)$ using \eqref{eqn: one containmentA}.
For each fixed $k$, after passing to a diagonal subsequence, we have $y_{i,a}\to y_a \in B_\infty$ and $t_{a,i}\to t_a \in [0,\e^{2k/n}]$ for each $a\in\mathbb{N}$. (Here the convergence of points is meant with respect to the metric on $X\amalg (\amalg M_i)$ in which the Hausdorff convergence of the spaces occurs.) Define the set $\mathcal{C}^k = \{ y_a \in B_\infty: t_a >0\}$, and  define 
\begin{equation}\label{eqn: def of Ek}
\F^k = B_\infty \setminus \bigcup_{y_a \in \mathcal{C}^k} B_d(x_a, t_a^{1/2}).
\end{equation}
Since $\F^k$ is a Borel set, it is $m$-measurable and $\Hi^{n}$-measurable.
 Observe that $\sum_{a\in \mathcal{C}^k} t_a^{n/2}\leq \e^k$, and thus applying \eqref{eqn: limit space volumes of balls}, we find that
\begin{equation*}
	m(B_\infty\setminus \F^k) \leq  (1+\delta)\om_n \e^k.
\end{equation*} 
Thanks to property (3) of the theorem, we also get $\Hi^n (B_\infty\setminus \F^k) \leq (1+\delta)^2 \om_n \e^k$. So, in particular 
\begin{equation}
	\label{eqn: measure zero} 
\mathcal{H}^n (X \setminus \cup_k \mathcal{F}^k)= m (X \setminus \cup_k \mathcal{F}^k)  =0.
\end{equation}

We claim that $\psi|_{\F^k}$ is a $(1+\e)^k$ bi-Lipschitz map onto its image. 
In view of \eqref{eqn: measure zero} and the measurability of the $\F^k$, this will show that $B_\infty$ is $m$-rectifiable and $\Hi^n$-rectifiable, thereby completing the proof.
 To this end, for any $x,y \in \F^k$, we may find sequences $x_i,y_i \in M_i$ such that $x_i\to x$ and  $y_i\to y$ and with $d_i:= d_{g_i}(x_i,y_i)\to d:=d(x,y)$. 
 From the definition of $\F^k$, we see that for $i$ sufficiently large, either $x_i \in \F^k_i$, or else $x_i \in B_{g_i}(y_{a,i}, t_{a,i}^{1/2})$ with $t_{a,i}\to 0$ as $i\to\infty$. The same holds for $y_i$. In either case, $d_{g_i}(x_i,y_i) \geq \e \min\{ r_k(x_i),r_k(y_i)\}$ for $i$ sufficiently large, where $r_k$ is defined in \eqref{eqn: r_k def} above. The claim then follows from Lemma~\ref{rmk: content bound replacement}, Remark~\ref{rmk: A}, and the Arzel\`{a}-Ascoli  theorem. This completes the proof.
 \end{proof}

\bibliographystyle{alpha}
\bibliography{epsReg.bib}

\begin{thebibliography}{HKST15}

\bibitem[All21]{BrianAllen}
B.~Allen.
\newblock From ${L}^p$ bounds to {G}romov-{H}ausdorff convergence of
  {R}iemannian manifolds.
\newblock {\em Preprint available at ar{X}iv:2106.14231}, 2021.

\bibitem[APS20]{APS}
B.~Allen, R.~Perales, and C.~Sormani.
\newblock Volume above distance below.
\newblock {\em To appear in J. Diff. Geom., preprint available at
  ar{X}iv:2003.01172}, 2020.

\bibitem[AS20]{AllenSormani}
B.~Allen and C.~Sormani.
\newblock Relating notions of convergence in geometric analysis.
\newblock {\em Nonlinear Anal.}, 200:111993, 33, 2020.

\bibitem[Bam16]{Bamler1}
R.~H. Bamler.
\newblock A {R}icci flow proof of a result by {G}romov on lower bounds for
  scalar curvature.
\newblock {\em Math. Res. Lett.}, 23(2):325--337, 2016.

\bibitem[BG]{PBG2}
P.~Burkhardt-Guim.
\newblock {A}{D}{M} mass for ${C}^0$ metrics and distortion under
  {R}icci-{D}e{T}urck flow.
\newblock {\em Preprint at arXiv:2208.14550}.

\bibitem[BG19]{PBG1}
P.~Burkhardt-Guim.
\newblock Pointwise lower scalar curvature bounds for {$C^0$} metrics via
  regularizing {R}icci flow.
\newblock {\em Geom. Funct. Anal.}, 29(6):1703--1772, 2019.

\bibitem[BZ17]{BZ1}
R.~H. Bamler and Q.~S. Zhang.
\newblock Heat kernel and curvature bounds in {R}icci flows with bounded scalar
  curvature.
\newblock {\em Adv. Math.}, 319:396--450, 2017.

\bibitem[CC96]{CCannals}
J.~Cheeger and T.~H. Colding.
\newblock Lower bounds on {R}icci curvature and the almost rigidity of warped
  products.
\newblock {\em Ann. of Math. (2)}, 144(1):189--237, 1996.

\bibitem[CC97]{CC1}
J.~Cheeger and T.~H. Colding.
\newblock On the structure of spaces with {R}icci curvature bounded below. {I}.
\newblock {\em J. Differential Geom.}, 46(3):406--480, 1997.

\bibitem[CC00a]{CC2}
J.~Cheeger and T.~H. Colding.
\newblock On the structure of spaces with {R}icci curvature bounded below.
  {II}.
\newblock {\em J. Differential Geom.}, 54(1):13--35, 2000.

\bibitem[CC00b]{CC3}
J.~Cheeger and T.~H. Colding.
\newblock On the structure of spaces with {R}icci curvature bounded below.
  {III}.
\newblock {\em J. Differential Geom.}, 54(1):37--74, 2000.

\bibitem[Che22]{LiangCheng}
L.~Cheng.
\newblock On pseudo-locality theorems of {R}icci flows on incomplete manifolds.
\newblock {\em Preprint available at ar{X}iv:2210.15397}, 2022.

\bibitem[CJN21]{CNJ}
J.~Cheeger, W.~Jiang, and A.~Naber.
\newblock Rectifiability of singular sets of noncollapsed limit spaces with
  {R}icci curvature bounded below.
\newblock {\em Ann. of Math. (2)}, 193(2):407--538, 2021.

\bibitem[CLN06]{CLNbook}
B.~Chow, P.~Lu, and L.~Ni.
\newblock {\em Hamilton's {R}icci flow}, volume~77 of {\em Graduate Studies in
  Mathematics}.
\newblock American Mathematical Society, Providence, RI; Science Press Beijing,
  New York, 2006.

\bibitem[CN13]{CN}
J.~Cheeger and A.~Naber.
\newblock Lower bounds on {R}icci curvature and quantitative behavior of
  singular sets.
\newblock {\em Invent. Math.}, 191(2):321--339, 2013.

\bibitem[Col97]{Colding1}
T.~H. Colding.
\newblock Ricci curvature and volume convergence.
\newblock {\em Ann. of Math. (2)}, 145(3):477--501, 1997.

\bibitem[CW12]{ChenWang}
X.~Chen and B.~Wang.
\newblock Space of {R}icci flows {I}.
\newblock {\em Comm. Pure Appl. Math.}, 65(10):1399--1457, 2012.

\bibitem[Gro75]{Gross}
L.~Gross.
\newblock Logarithmic {S}obolev inequalities.
\newblock {\em Amer. J. Math.}, 97(4):1061--1083, 1975.

\bibitem[Ham95]{HamiltonRFsurvey}
R.~S. Hamilton.
\newblock The formation of singularities in the {R}icci flow.
\newblock In {\em Surveys in differential geometry, {V}ol. {II} ({C}ambridge,
  {MA}, 1993)}, pages 7--136. Int. Press, Cambridge, MA, 1995.

\bibitem[HKST15]{MMSBOOK}
J.~Heinonen, P.~Koskela, N.~Shanmugalingam, and J.~T. Tyson.
\newblock {\em Sobolev spaces on metric measure spaces}, volume~27 of {\em New
  Mathematical Monographs}.
\newblock Cambridge University Press, Cambridge, 2015.
\newblock An approach based on upper gradients.

\bibitem[HN14]{HeinNaber}
H.-J. Hein and A.~Naber.
\newblock New logarithmic {S}obolev inequalities and an {$\epsilon$}-regularity
  theorem for the {R}icci flow.
\newblock {\em Comm. Pure Appl. Math.}, 67(9):1543--1561, 2014.

\bibitem[LNNa]{LNNSurvey}
M.-C. Lee, A.~Naber, and R.~Neumayer.
\newblock Convergence and regularity of manifolds with scalar curvature and
  entropy lower bounds.
\newblock {\em To appear in Perspectives in Scalar Curvature}.

\bibitem[LNNb]{LNN1}
M.-C. Lee, A.~Naber, and R.~Neumayer.
\newblock $d_p$ {C}onvergence and $\epsilon$-regularity theorems for entropy
  and scalar curvature lower bounds.
\newblock {\em To appear in Geom. Topol.}

\bibitem[LT22]{LeeTopping}
M.-C. Lee and P.~M. Topping.
\newblock Metric limits of manifolds with positive scalar curvature, 2022.

\bibitem[Per02]{perelman2002entropy}
G.~Perelman.
\newblock The entropy formula for the {R}icci flow and its geometric
  applications.
\newblock {\em arXiv preprint math/0211159}, 2002.

\bibitem[Pet98]{petersenBook}
P.~Petersen.
\newblock {\em Riemannian geometry}, volume 171 of {\em Graduate Texts in
  Mathematics}.
\newblock Springer-Verlag, New York, 1998.

\bibitem[Shi89]{ShiExistence}
W.-X. Shi.
\newblock Deforming the metric on complete {R}iemannian manifolds.
\newblock {\em J. Differential Geom.}, 30(1):223--301, 1989.

\bibitem[Sim83]{simonGMT}
L.~Simon.
\newblock {\em Lectures on geometric measure theory}, volume~3 of {\em
  Proceedings of the Centre for Mathematical Analysis, Australian National
  University}.
\newblock Australian National University, Centre for Mathematical Analysis,
  Canberra, 1983.

\bibitem[TW15]{TianWang}
G.~Tian and B.~Wang.
\newblock On the structure of almost {E}instein manifolds.
\newblock {\em J. Amer. Math. Soc.}, 28(4):1169--1209, 2015.

\bibitem[Wan18]{BingWangA}
B.~Wang.
\newblock The local entropy along {R}icci flow---{P}art {A}: the
  no-local-collapsing theorems.
\newblock {\em Camb. J. Math.}, 6(3):267--346, 2018.

\bibitem[Wan20]{BingWangB}
B.~Wang.
\newblock The local entropy along {R}icci flow---{P}art {B}: the
  pseudo-locality theorems, 2020.

\end{thebibliography}
\end{document}